\documentclass[11pt]{amsart}
\usepackage{amsmath,amsthm,amsfonts,amssymb,amscd,flafter,pinlabel}
\usepackage[mathscr]{eucal}
\usepackage{graphics}
 \usepackage[all]{xy}
\usepackage{caption, subcaption}
\usepackage[usenames,dvipsnames]{color}
\usepackage{graphicx}

\usepackage{pgf}
\usepackage{tikz}
\usetikzlibrary{arrows,automata}
\usepackage[latin1]{inputenc}

\usepackage[colorlinks=true, urlcolor=NavyBlue, linkcolor=NavyBlue, citecolor=NavyBlue, pdftitle={Twisting and knot Floer homology}, pdfauthor={Peter Lambert-Cole}, pdfsubject={Knot Floer Homology}, pdfkeywords={Heegaard Floer homology, 57M27, 57R58}]{hyperref}

\newtheorem{theorem}{Theorem}[section]
\newtheorem{corollary}[theorem]{Corollary}

\newtheorem{proposition}[theorem]{Proposition}

\newtheorem{lemma}[theorem]{Lemma}

\theoremstyle{definition}

\theoremstyle{definition}
\newtheorem{remark}[theorem]{Remark}

\newcommand\HFK{{\rm {HFK}}}
\newcommand\HFKh{{\rm {\widehat{HFK}}}}

\newcommand\HFKm{{\rm {HFK^{-}}}}

\newcommand\CFKh{{\rm {\widehat{CFK}}}}
\newcommand\CFKm{{\rm {CFK^{-}}}}

\newcommand\alphas{\mbox{\boldmath$\alpha$}}
\newcommand\betas{\mbox{\boldmath$\beta$}}
\newcommand\Sym{{\rm {Sym}}}
\newcommand\ws{\mathbf w}
\newcommand\zs{\mathbf z}
\newcommand\FF{\mathbb F}
\newcommand\M{\mathcal M}

\newcommand\TT{\mathbb{T}}
\newcommand\tori{\TT_{\alpha} \cap \TT_{\beta}}
\newcommand\SH{\mathcal{H}}
\newcommand\rk{\text{rk }}

\def\x{\mathbf{x}}

\def\y{\mathbf{y}}
\def\z{\mathbf{z}}
\def\w{\mathbf{w}}

\newcommand{\mr}{\mathrm}

\newcommand{\cM}{\mathcal{M}}

\newcommand{\CC}{\mathbb{C}}
\newcommand{\RR}{\mathbb{R}}
\newcommand{\del}{\partial}

\newcommand{\ZZ}{\mathbb{Z}}

\newcommand{\cL}{\mathcal{L}}
\newcommand{\cJ}{\mathcal{J}}

\newcommand{\DD}{\mathbb{D}}

\begin{document}

\title[{Twisting, mutation and knot Floer homology}]{Twisting, mutation and knot Floer homology}

\author[P. Lambert-Cole]{Peter Lambert-Cole}
\address{Department of Mathematics \\ Indiana University}
\email{pblamber@indiana.edu}
\urladdr{\href{https://www.pages.iu.edu/~pblamber/}{https://www.pages.iu.edu/\~{}pblamber}}

% \date{\today}
\maketitle

%%%%%%%%%%%%%%%%%%%%%%%%%%%%%%%%%%%%%%%%%%%%%%%%%%%%%%%

%%%%%%%%%%%%%%%%%%%%%%%%%%%%%%%%%%%%%%%%%%%%%%%%%%%%%%%
\begin{abstract}
%%%%%%%%%%%%%%%%%%%%%%%%%%%%%%%%%%%%%%%%%%%%%%%%%%%%%%%
Let $\cL$ be a knot with a fixed positive crossing and $\cL_n$ the link obtained by replacing this crossing with $n$ positive twists.  We prove that the knot Floer homology $\HFKh(\cL_n)$ `stabilizes' as $n$ goes to infinity.  This categorifies a similar stabilization phenomenon of the Alexander polynomial.  As an application, we construct an infinite family of prime, positive mutant knots with isomorphic bigraded knot Floer homology groups.  Moreover, given any pair of positive mutants, we describe how to derive a corresponding infinite family of positive mutants with isomorphic bigraded $\HFKh$ groups, Seifert genera, and concordance invariant $\tau$.

%%%%%%%%%%%%%%%%%%%%%%%%%%%%%%%%%%%%%%%%%%%%%%%%%%%%%%%	
\end{abstract}
%%%%%%%%%%%%%%%%%%%%%%%%%%%%%%%%%%%%%%%%%%%%%%%%%%%%%%%

%%%%%%%%%%%%%%%%%%%%%%%%%%%%%%%%%%%%%%%%%%%%%%%%%%%%%%%
\section{Introduction} % (fold)
\label{sec:introduction}
%%%%%%%%%%%%%%%%%%%%%%%%%%%%%%%%%%%%%%%%%%%%%%%%%%%%%%%

An interesting open question is the relationship between mutation and knot Floer homology.  While many knot polynomials and homology theories are insensitive to mutation, the bigraded knot Floer homology groups can detect mutation \cite{OS-mutation} and genus 2 mutation \cite{Moore-Starkston}.  Conversely, explicit computations \cite{BG-computations} and a combinatorial formulation \cite{BL-spanning} suggest that the $\delta$-graded $\HFKh$ groups are mutation-invariant.

In this paper, we construct an infinite family of prime, positive mutants whose $\HFKh$ groups agree.  Recall that a mutation is {\it positive} if the mutant links admit compatible orientations.

\begin{theorem}
\label{thrm:pos-mutants-hfk}
There exist an infinite family positively mutant knots $\{KT_n,C_n\}$ such that 
\begin{enumerate}
\item for all $n$, the knots $KT_n$ and $C_n$ are not isotopic,
\item for all $|n| \gg 0$, the knots $KT_n,C_n$ are prime, and
\item for all $|n| \gg 0$, there is a bigraded isomorphism
\[\HFKh(KT_n) \cong \HFKh(C_n)\]
\end{enumerate}
\end{theorem}

The families $\{KT_n\}$ and $\{C_n\}$ are constructed from the Kinoshita-Teraska and Conway knots, respectively, by adding $n$ full twists to the knots just outside the mutation sphere.  Note that while $KT$ and $C$ are distinguished by $\HFKh$ \cite{OS-mutation, BG-computations} and their genera \cite{Gabai-genera}, adding sufficiently many twists forces their bigraded knot Floer groups (and therefore genera) to coincide.  To distinguish $KT_n$ and $C_n$, we use tangle invariants introduced by Cochran and Ruberman \cite{Cochran-Ruberman}.  

The bigraded mutation invariance of $\HFKh$ follows from a more general fact.  In a particular sense, `most' positive mutants have isomorphic knot Floer homology groups over $\FF_2$.  Moreover, we can prove that `most' positive mutants have the same concordance invariant $\tau$.

\pagebreak

\begin{theorem}
\label{thrm:mutants-isom-hfk}
Let $\cL$ be an oriented knot with positive mutant $\cL'$.  Let $\cL_n, \cL'_n$ denote the mutant links obtained by adding $n$ half-twists along parallel-oriented strands of $\cL$ just outside the mutation sphere.  Then for $|n| \gg 0$ there is a bigraded isomorphism
\[\HFKh(\cL_n) \cong \HFKh(\cL'_n)\]
Moreover, for $|n| \gg 0$,
\begin{align*}
g(\cL_{2n}) &= g(\cL'_{2n}) \\
\tau(\cL_{2n}) &= \tau(\cL'_{2n})
\end{align*}
where $g$ denotes the Seifert genus and $\tau$ denotes the knot Floer concordance invariant.
\end{theorem}

The proof of Theorem \ref{thrm:mutants-isom-hfk} follows from a second result we prove in this paper.  We show that knot Floer homology categorifies the following phenomenon for the Alexander polynomial.  Let $\cL \subset S^3$ be a link with a positive crossing.  For a positive integer $n \in \ZZ$, let $\cL_n$ be the link obtained by replacing the crossing with $n$ positive twists.  As $n$ goes to infinity, the Alexander polynomial of $\cL_n$ stabilizes to the form
\[\Delta_{\cL_n}(t) = t^{\frac{n-k}{2}} \cdot f(t)+ d \cdot \Delta_{T(2,n-k)}(t) +  t^{-\frac{n-k}{2}} \cdot f(t^{-1}) \]
for some integers $k,d$ and some polynomial $f(t) \in \ZZ[\sqrt{t}]$.  In this paper, we show that the knot Floer homology of $\cL_n$ stabilizes in a similar fashion as $n$ goes to infinity.

\begin{theorem}
\label{thrm:stabilization}
Let $\cL = \cL_1$ be an oriented knot with a positive crossing and let $\cL_n$ be the link obtained by replacing this crossing with $n$ positive half-twists.  Without loss of generality, assume that at the crossing, the two strands belong to the same component of $\cL_1$.

\begin{enumerate}
\item There exists some $k > 0$ such that for $|n|$ sufficiently large, the knot Floer homology of $\cL_n$ satisfies
\begin{align*}
\HFKh(\cL_n,j) &\cong \HFKh(\cL_{n+2},j+1) && \text{for $j \geq -k$}\\
\HFKh(\cL_{n},j) &\cong \HFKh(\cL_{n+2},j-1)[2] && \text{for $j \leq k$}
\end{align*}
where $[i]$ denotes shifting the Maslov grading by $i$.

\item There exists a positive integer $k$ and bigraded vector spaces $\widehat{F}_{\circ},\widehat{F}_{\bullet},\widehat{A},\widehat{B}$ such that for odd $n$ sufficiently large
\begin{align*}
\label{formula:hfk-stabilized}
\HFKh(\cL_n) &\cong \widehat{F}_{\circ}[2(n-k),(n-k)] \oplus \bigoplus_{i = k-n}^{0} \widehat{A}[2(n-k) + 2i, (n-k)+2i] \\
& \, \, \oplus \bigoplus_{i = k-n+1}^{0} \widehat{B}[2(n-k) + 2i - 1,(n-k)+2i -1] \oplus \widehat{F}_{\bullet}[0,(k-n)]
\end{align*}
where $[i,j]$ denotes shifting the homological grading by $i$ and Alexander grading by $j$ and 
\begin{enumerate}
\item $\widehat{F}_{\bullet}$ (resp. $\widehat{F}_{\circ}$) is supported in nonnegative (resp. nonpositive) Alexander gradings,
\item $\widehat{A},\widehat{B}$ are supported in Alexander grading 0.
\end{enumerate}

\item For even $n$ sufficiently large
\begin{align*}
\HFKh(\cL_n) &\cong \HFKh(\cL_{n+1}) \oplus \HFKh(\cL_{n-1})[1]
\end{align*}
where the summands on the right are described in Part (2) and $[i]$ denotes shifting the Maslov grading.
\end{enumerate}
\end{theorem}

A key consequence (Lemma \ref{lemma:skein-maximal}) of Theorem \ref{thrm:stabilization} is that, for $|n|$ sufficiently large, the skein exact triangle for the triple $(\cL_{n+1},\cL_{n-1},\cL_n)$ splits.  We apply this observation to prove Theorem \ref{thrm:mutants-isom-hfk}.

\subsection{Acknowledgements}

I would like to thank Matt Hedden, Adam Levine, Tye Lidman, Allison Moore and Dylan Thurston for their interest and encouragement.  I would also like to thank the organizers of the "Topology in dimension 3.5" conference, as well as Kent Orr, for helping me learn about Tim Cochran's work.

\section{Twisting and knot Floer homology}

Throughout this section, let $\cL_1 \subset S^3$ be an oriented link with a distinguished positive crossing and for any $n \in \ZZ$ let $\cL_n$ denote the link obtained by replacing this crossing with $n$ half twists.

\subsection{Alexander polynomial}

Let $\Delta_{\cL}(t)$ denote the symmetrized Alexander polynomial of $\cL$ and let $\Delta_k(t)$ denote the symmetrized Alexander polynomial of the $(2,k)$-torus link.

\begin{proposition}
\label{prop:twist-Alex-relation}
The Alexander polynomial of $\cL_n$ satisfies the relation
\[ \Delta_{\cL_n} = \Delta_{n+1} \Delta_{\cL_{0}} + \Delta_{n} \Delta_{\cL_{-1}}\]
\end{proposition}

\begin{proof}
The formula is correct when $n = 0,-1$ since
\begin{align*}
\Delta_{\cL_0} &= \Delta_1 \Delta_{\cL_0} + \Delta_0 \Delta_{\cL_{-1}}  & \Delta_{\cL_{-1}} &= \Delta_0 \Delta_{\cL_0} + \Delta_{-1} \Delta_{\cL_{-1}} \\
&= \Delta_{\cL_0} & &= \Delta_{\cL_{-1}}
\end{align*}
For positive $n$, the formula follows inductively by using the oriented skein relation for the Alexander polynomial:
\begin{align*}
\Delta_{\cL_{n+1}} &= \Delta_2 \cdot \Delta_{\cL_n} + \Delta_{\cL_{n-1}} \\
& =\Delta_2 \left( \Delta_{n+1} \Delta_{\cL_0} + \Delta_{n} \Delta_{\cL_{-1}} \right) + \Delta_{n} \Delta_{\cL_0} + \Delta_{n-1} \Delta_{\cL_{-1}} \\
&= (\Delta_2 \Delta_{n+1} + \Delta_n) \Delta_{\cL_0} + (\Delta_2 \Delta_{n} + \Delta_{n-1}) \Delta_{\cL_0} \\
& = \Delta_{n+2} \Delta_{\cL_0} + \Delta_{n+1} \Delta_{\cL_{-1}}
\end{align*}
The first line is the oriented skein relation, the second is obtained by using the induction hypothesis, the third is obtained by rearranging the terms and the fourth is obtained by applying the oriented skein relation to the Alexander polynomials of torus knots.  A similar inductive argument proves the formula for $n < -1$.
\end{proof}

\begin{corollary}
\label{cor:twisting-Alex}
There exists some $k > 0$, some $d \in \ZZ$, and some polynomial $f(t) \in \ZZ[\sqrt{t}]$ such that for $n$ sufficiently large, the Alexander polynomial of $\cL_n$ has the form
\[\Delta_{\cL_n}(t) = t^{\frac{n-k}{2}} \cdot f(t)+ d \cdot \Delta_{T(2,n-k)}(t) +  t^{-\frac{n-k}{2}} \cdot f(t^{-1}) \]

\end{corollary}

\begin{proof}
The Alexander polynomial of $T(2,k)$ is either
\[ \Delta_k(t) = \sum_{s = -\frac{k-1}{2}}^{\frac{k-1}{2}} (-t)^s \qquad \text{ or } \qquad \Delta_k(t) = t^{\frac{1}{2}}\sum_{s = -\frac{k}{2}}^{\frac{k}{2} -1} (-t)^s \]
according to whether $k$ is odd or even. Without loss of generality, assume that $\Delta_{\cL_1} \in \ZZ[t^{\pm 1}]$ and $n$ is odd.  Suppose that the Alexander polynomials of $\cL_1,\cL_0,\cL_n, \cL_{n+2}$ have the forms
\begin{align*}
\Delta_{\cL_1} &=\sum_{s=-k_0}^{k_0} a_s t^s  &\Delta_{\cL_0} &= t^{-\frac{1}{2}} \cdot \sum_{s = -k_1}^{k_1+1} b_s t^{s} \\
 \Delta_{\cL_n} &= \sum_{s = -k_n}^{k_n} c_s t^s &  \Delta_{\cL_{n+2}} &= \sum_{s = -k_{n+2}}^{k_{n+2}} d_s t^s
\end{align*}
Applying Proposition \ref{prop:twist-Alex-relation}, a straightforward computation shows that the coefficents $\{c_s\}$ are an alternating sum of some subsets of the coefficients $\{a_s\}$ and $\{b_s\}$.  For $n$ sufficiently large and $s \geq 0$ the coefficients of $\Delta_{\cL_n}$ and $\Delta_{\cL_{n+2}}$ satisfy the relation $c_s = d_{s+1}$.  This implies that $\Delta_{\cL_n}$ has the specified form.
\end{proof}

\begin{corollary}
\label{cor:degree-Alex}
If $\cL_1$ is a knot, then the degree of $\Delta_{\cL_n}$ goes to infinity as $|n|$ goes to infinity.
\end{corollary}

\begin{proof}
Since $\cL_1$ is a knot, the determinant $|\Delta_{\cL_n}(-1)|$ is nonzero when $n$ is odd.  Thus, in the formula from Corollary \ref{cor:twisting-Alex} the polynomial $f(t)$ and integer $d$ cannot both be $0$.  As a result, the degree of $\Delta_{\cL_n}$ is at least $n - k$ for some $k$ independent of $n$.
\end{proof}

\subsection{Knot Floer homology}

A multi-pointed Heegaard diagram $\SH = (\Sigma,\alphas,\betas,\z,\w)$ for a link $L \subset S^3$ is a tuple consisting of a genus $g$ Riemann surface $\Sigma$, two multicurves $\alphas = \{ \alpha_1 \cup \dots \cup \alpha_{g+n} \}$ and $\betas = \{ \beta_1 \cup \dots \cup \beta_{g+n}\}$, and two collections of basepoints $\z = \{z_1,\dots, z_{n-1} \}$ and $\w = \{w_1,\dots,w_{n-1}\}$ such that
\begin{enumerate}
\item $(\Sigma, \alphas,\betas)$ is a Heegaard diagram for $S^3$,
\item each component of $\Sigma \setminus \alphas$ and $\Sigma \setminus \betas$ contains exactly 1 $\z$-basepoint and 1 $\w$-basepoint,
\item the basepoints $\zs,\ws$ determine the link $L$ as follows: choose collections of embedded arcs $\{\gamma_1,\dots, \gamma_{n-1} \}$ in $\Sigma - \alphas$ and $\{\delta_1,\dots,\delta_{n-1}\}$ in $\Sigma \setminus \betas$ connecting the basepoints.  Then after depressing the arcs $\{\gamma_i\}$ into the $\alpha$-handlebody and the arcs $\{\delta_j\}$ into the $\beta$-handlebody, their union is $L$.
\end{enumerate}

From a multipointed multipointed Heegaard diagram $\SH$ we obtain a complex $\CFKm(\SH)$.  In the symmetric product $\Sym^{g+n}(\Sigma)$ of the Heegaard surface, the multicurves $\alphas,\betas$ determine $(g+n)$-dimensional tori $\TT_\alpha = \alpha_1 \cup \dots \cup \alpha_{g+n}$ and $\TT_\beta = \beta_1 \cup \dots \cup \beta_{g+n}$.  The knot Floer complex $\CFKh(\SH)$ is freely generated over $\FF[U_{w_2},\dots,U_{w_{n-1}}]$ by the intersection points $\mathfrak{G}(\SH) = \TT_{\alpha} \cap \TT_{\beta}$.  

The complex possesses two gradings, the Maslov grading and the Alexander grading.  Given any Whitney disk $\phi \in \pi_2(\x,\y)$ connecting the two generators, the relative gradings of two generators $\x,\y \in \TT_{\alpha} \cap \TT_{\beta}$ satisfy the formulas
\begin{align*}
M(\x) - M(\y) &= \mu(\phi) - 2 \sum n_{w_i}(\phi) \\
A(\x) - A(\y) &= \sum n_{z_i}(\phi) - \sum n_{w_i}(\phi)
\end{align*}
where $\mu(\phi)$ is the Maslov index of $\phi$ and $n_{z_i}(\phi)$ and $n_{w_i}(\phi)$ denote the algebraic intersection numbers of $\phi$ with respect to the subvarieties $\{w_i\} \times \Sym^{g+n-1}(\Sigma)$ and $\{z_i\} \times \Sym^{g+n-1}(\Sigma)$.  Note that while there are formulations of link Floer homology with an independent Alexander grading for each link component, we restrict to a single Alexander grading.  In addition, multiplication by any formal variable $U_{w_i}$ decrease the Maslov grading by 2 and the Alexander grading by 1.  The graded Euler characteristic of the complex is a multiple of the Alexander polynomial of the link and satisfies the formula
\[(t^{\frac{1}{2}} - t^{-\frac{1}{2}})^{l-1} \Delta_L(t) = \sum_{j \in \ZZ} t^j \cdot \left( \sum_{i \in \ZZ} (-1)^i \text{ dim}_{\FF} \CFKh(\SH)_{i,j} \right)\]
where $l$ is the number of components of $L$.

The differential is defined by certain counts of pseudoholomorphic curves in $\Sym^{g+n}(\Sigma)$.  We review some of the analytic details from \cite{OS-HF}.  Fix a complex structure $j$ and Kahler form $\eta$ on $\Sigma$ and let $J$ be the induced complex structure on $\Sym^{g+n}(\Sigma)$.  Together with the basepoints $\zs$, these determine a set $\cJ(j,\eta,V_{\zs})$ of almost-symmetric complex structures on $\Sym^{g+n}(\Sigma)$.  Set $\DD = [0,1] \times i \RR \subset \CC$. If $j$ is generic, then we can choose a generic path $J_s$ in a small neighborhood of $J \in \cJ(j,\eta,V_{\zs})$ such that for any pair $\x,\y \in \TT_{\alpha} \cap \TT_{\beta}$ and any $\phi \in \pi_2(\x,\y)$, the moduli space of maps

\[ \mathcal{M}_{J_s}(\phi) := \left\{ u: \mathbb{D} \rightarrow \Sym^N(S^2) \middle| \begin{array}{r l r l}
u(\{1\} \times \RR) & \subset \TT_{\alpha} & \lim_{t \rightarrow -\infty} u(s + it) &= \x \\
u(\{0\} \times \RR) & \subset \TT_{\beta} & \lim_{t \rightarrow \infty} u(s + it) &= \y \\
\frac{du}{ds} + J_s \frac{du}{dt} &= 0  & [u] &= \phi \end{array} \right\} \]
is transversely cut out.  Translation in $\DD$ induces an $\RR$-action on $\cM_{J_s}(\phi)$ and the unparametrized moduli space $\widehat{\cM}_{J_s}(\phi)$ is the quotient space of this action.  If $\mu(\phi) = 1$ then $\widehat{\cM}_{J_s}(\phi)$ is a compact 0-manifold consisting of a finite number of points.

Let $J_s$ be a generic path of almost-symmetric complex structures.  Define

\[
	\widehat{\del} \x = \sum_{\y \in \TT_\alpha \cap \TT_\beta} \sum_{\substack{\phi \in \pi_2(\x,\y),\\ \mu(\phi) = 1,\\ n_\z(\phi) = 0, \\ n_{w_1}(\phi) = 0}} \# \widehat{\M}_{J_s}(\phi) \cdot U_{w_2}^{n_{w_2}(\phi)} \dots U_{w_{n-1}}^{n_{w_{n-1}}(\phi)} \cdot \y,
\]

The ``hat'' version of the knot Floer homology of $L$ is the homology of the complex $(\CFKh(\SH),\widehat{\del})$:
\[
	\HFKh_*(L) := \mr{H}_*(\CFKh(\SH)),
\]

It is independent of the choice of Heegaard diagram $\SH$ encoding $L$ or the path $J_{s}$ of almost-complex structures.

Knot Floer homology  satifies a skein exact triangle that categorifies the skein relation for the Alexander polynomial. 

\begin{theorem}[Skein exact triangle \cite{OS-HFK,OS-skein,OSS-book}]
Let $\cL_+,\cL_-,\cL_0$ be three links that differ at a single crossing.  If the two strands of $\cL_+$ meeting at the crossing lie in the same component, then there is an Alexander grading-preserving exact triangle

\[ \xymatrix{
\HFKh(\cL_0) \ar[rr]^{\widehat{f}} && \HFKh(\cL_+) \ar[dl]^{\widehat{g}} \\
& \HFKh(\cL_-) \ar[ul]^{\widehat{h}[-1]} & 
} \]

If the two strands of $\cL_+$ meeting at the crossing lie in two different components, then there is an Alexander grading-preserving exact triangle

\[ \xymatrix{
\HFKh(\cL_0)\otimes V \ar[rr]^{\widehat{f}} && \HFKh(\cL_+) \ar[dl]^{\widehat{g}} \\
& \HFKh(\cL_-) \ar[ul]^{\widehat{h}[-1]} & 
} \]

where $V$ is a bigraded vector space satisfying
\[V_{m,s} = \begin{cases} \FF^2 & \text{if }(m,s) = (-1,0) \\ \FF & \text{if }(m,s) = (0,1) \text{ or }(-2,-1) \\ 0 & \text{otherwise} \end{cases}\]
\end{theorem}

\subsection{A multi-pointed Heegaard diagram for $\cL_n$}

Choose an $N+1$-bridge presentation for $\cL_1$ so that near its distinguished crossing it has the form in the top right of Figure \ref{fig:crossing-bridge}.  The bridge presentation consists of $N+1$ bridges $\{a_0,\dots,a_N\}$ and $N+1$ overstrands $\{b_0,\dots,b_N\}$.  Label the arcs in the bridge presentation so that $a_1$ and $a_0$ are the left and right bridges in Figure \ref{fig:crossing-bridge} and the overstrands $b_1$ and $b_0$ share endpoints with $a_1$ and $a_0$, respectively.  We can also obtain a bridge presentation for $\cL_0$, the link obtained by taking the 0-resolution of the distinguished crossing, in the top left of Figure \ref{fig:crossing-bridge}.  Let $z_1$ and $z_0$ denote the endpoints of the left and right bridges, respectively, in the local picture of the crossing.

From these initial bridge presentations, we can obtain a bridge presentation for $\cL_n$.  Let $\gamma$ be a curve containing the points $z_1$ and $z_0$.  Orient $\gamma$ counter-clockwise, as the boundary of a disk containing the two points.  If $n$ is even, apply $\frac{n}{2}$ negative Dehn twists along $\gamma$ to the arcs $b_1$ and $b_0$ to obtain a bridge presentation for $\cL_n$.  If $n$ is odd, apply $\frac{n-1}{2}$ negative Dehn twists.

We can obtain a multipointed Heegaard diagram encoding $\cL_n$ from its bridge presentation.  For $i = 1,\dots,N$, let $\alpha_i$ be the boundary of a tubular neighborhood of the arc $a_i$ and let $\beta_i$ be the boundary of a tubular neighborhood of $b_i$.  Label the endpoints of the bridges as $z$- and $w$-basepoints so that the oriented boundary of $a_i$ is $w_i - z_i$.  Set $\zs = (z_0,z_1,\dots,z_N);$ $\ws = (w_0,w_1,\dots,w_N);$ $\alphas = (\alpha_1,\dots,\alpha_N);$ and $\betas = (\beta_1,\dots,\beta_N)$.  Let $\SH_n := (S^2,\alphas,\betas,\zs,\ws)$ denote this multipointed Heegaard diagram.  Note that locally, the diagrams $\SH_0$ and $\SH_1$ are identical.  In addition, the diagram $\SH_{n+2}$ can be obtained from the diagram $\SH_{n}$ by applying a negative Dehn twist along $\gamma$ to the multicurve $\betas$.

Let $T_{\gamma}: S^2 \rightarrow S^2$ denote the positive Dehn twist along $\gamma$ and let $T^*_{\gamma} : \Sym^N(S^2) \rightarrow \Sym^N(S^2)$ be the induced map.  If $\SH_n = (S^2,\alphas,\betas,\zs,\ws)$ is a multipointed Heegaard diagram for $\cL_n$ then $\SH_{n+2} = (S^2,\alphas, (T_{\gamma})^{-1}\betas,\zs,\ws)$.  The generators of $\CFKh(\SH_n)$ are $\tori$ and the generators of $\CFKh(\SH_{n+2})$ are $\TT_{\alpha} \cap (T^*_{\gamma})^{-1} \TT_{\beta}$.  The negative Dehn twist introduces 4 new intersection points between $\alpha_1$ and $\beta_1$ and does not destroy any.  Thus, there is a set injection from $\tori$ to $\TT_{\alpha} \cap (T^*_{\gamma})^{-1} \TT_{\beta}$.  If $\x \in \tori$ is a generator, let $\x'$ be the corresponding generator in $\TT_{\alpha} \cap (T^*_{\gamma})^{-1} \TT_{\beta}$.  The map $(T^*_{\gamma})^{-1}$ also induces a bijective map on Whitney disks from $\pi_2(\x,\y)$ to $\pi_2(\x',\y')$.  If $\phi \in \pi_2(\x,\y)$ is a Whitney disk, let $\phi'$ denote the corresponding Whitney disk.

We can partition the generators into 3 sets according to their vertex along $\alpha_1$.  The curves $\alpha_1,\beta_1$ intersect twice near the basepoint $z_1$ and we label these points $C,D$ as in Figure \ref{fig:crossing-heegaard}.  The $i^{\text{th}}$ negative Dehn twist along $\gamma$ introduces 4 new intersection points, which we label $a_i,b_i,c_i,d_i$.  See Figure \ref{fig:crossing-heegaard}.  Let $\x = (v_1,\dots,v_N)$ be a generator where $v_i \in \alpha_i$.  Define three sets:

\begin{align*}
\mathfrak{G}(\SH_n)_+ &:= \left\{ \x = (v_1,\dots,v_N) \, \middle| \, v_1 = C \text{ or } D \right\} \\
\mathfrak{G}(\SH_n)_{twist} &:= \left\{ \x = (v_1,\dots,v_N) \, \middle| \, v_1 \in \{a_i,b_i,c_i,d_i\} \text{ for some }i \right\} \\
\mathfrak{G}(\SH_n)_{-} &:= \mathfrak{G}(\SH_n) \smallsetminus \left( \mathfrak{G}(\SH_n)_+ \cup \mathfrak{G}(\SH_n)_{twist} \right)\\
\end{align*}

%%%%%%%%%%%%%%%%%%%%%%%%%%%%%%%%%%%%%%%%%%%%%%%
\begin{figure}
\centering
\labellist
	\Large\hair 2pt
	\pinlabel $\cL_{0}$ at 80 250
	\pinlabel $\gamma$ at 160 330
	\pinlabel $\cL_1$ at 600 250
	\pinlabel $\cL_2$ at 230 60
	\small\hair 2pt
	\pinlabel $a_1$ at 390 330
	\pinlabel $a_0$ at 690 330
	\pinlabel $b_0$ at 510 190
	\pinlabel $b_1$ at 575 190
	\pinlabel $b_1$ at 130 195
	\pinlabel $b_0$ at 195 195
	\pinlabel $a_1$ at 15 330
	\pinlabel $a_0$ at 310 330
	
\endlabellist
\includegraphics[width=.7\textwidth]{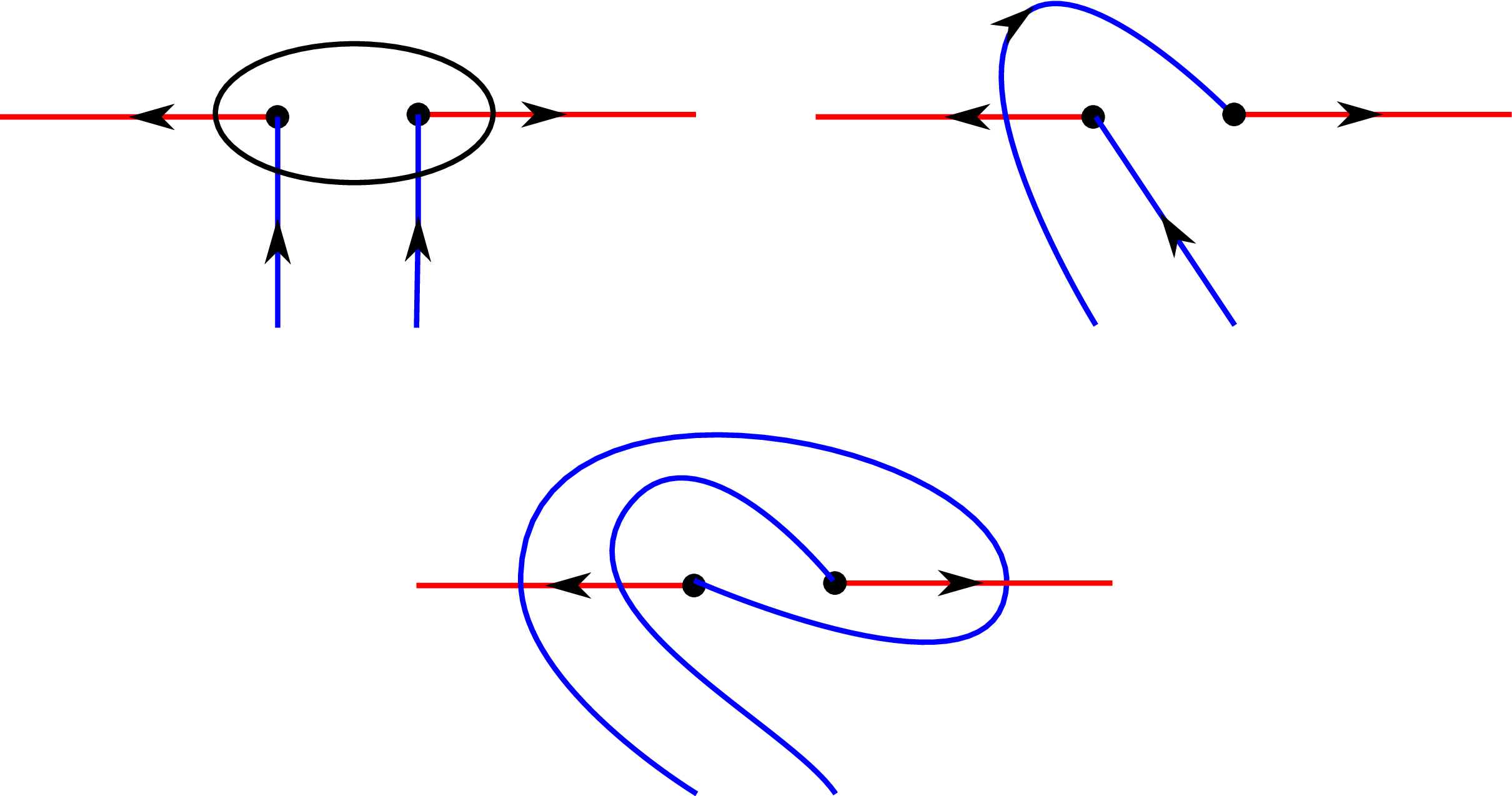}
\caption{Local pictures for $\cL_0,\cL_1,\cL_2$ in bridge position.  The diagram for $\cL_2$ can be obtained from the diagram for $\cL_0$ by applying a negative Dehn twist along $\gamma$ to the overstrands.}
\label{fig:crossing-bridge}
\end{figure}
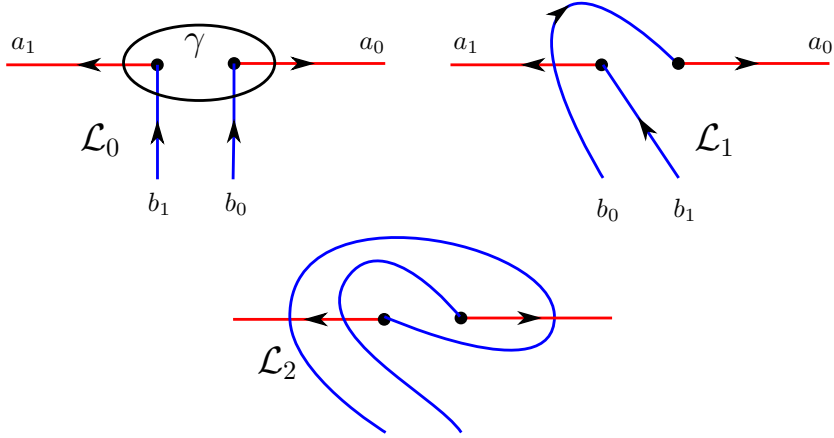
%%%%%%%%%%%%%%%%%%%%%%%%%%%%%%%%%%%%%%%%%%%%%%%%%

%%%%%%%%%%%%%%%%%%%%%%%%%%%%%%%%%%%%%%%%%%%%%%%
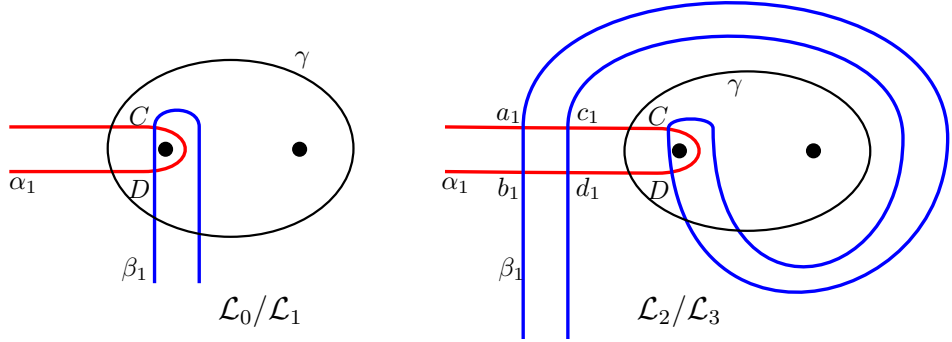
\begin{figure}
\centering
\labellist
	\large\hair 2pt
	\pinlabel $\cL_0/\cL_1$ at 180 20
	\pinlabel $\cL_2/\cL_3$ at 480 20
	\small\hair 2pt
	\pinlabel $\alpha_1$ at 10 110
	\pinlabel $\beta_1$ at 90 50
	\pinlabel $\gamma$ at 210 200
	\pinlabel $C$ at 93 160
	\pinlabel $D$ at 93 108
	\pinlabel $a_1$ at 358 160
	\pinlabel $b_1$ at 358 108
	\pinlabel $c_1$ at 415 160
	\pinlabel $d_1$ at 415 108
	\pinlabel $C$ at 465 160
	\pinlabel $D$ at 465 107
	\pinlabel $\alpha_1$ at 320 110
	\pinlabel $\beta_1$ at 360 50
	\pinlabel $\gamma$ at 520 180
\endlabellist
\includegraphics[width=.8\textwidth]{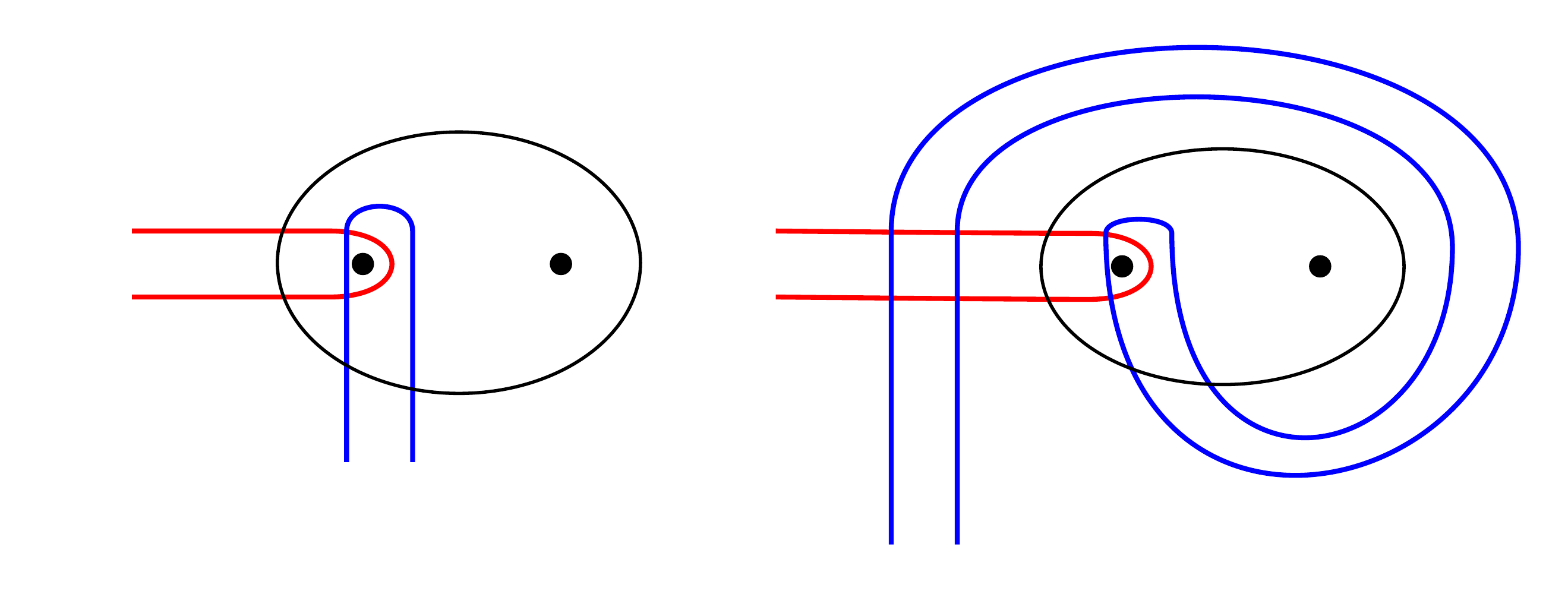}
\caption{Local pictures of the Heegaard diagram for $\cL_0$ or $\cL_1$ (on left) and $\cL_2$ or $\cL_3$ on right.  The Heegaard diagram on the right can be obtained from the left by applying a negative Dehn twist along $\gamma$ to $\betas$.  Both basepoints are $z$-basepoints.}
\label{fig:crossing-heegaard}
\end{figure}
%%%%%%%%%%%%%%%%%%%%%%%%%%%%%%%%%%%%%%%%%%%%%%%%%

\begin{lemma}
\label{lemma:homogeneous-grading}
Choose $\x,\y \in \mathfrak{G}(\SH_n)$ and $\phi \in \pi_2(\x,\y)$.  Let $\x',y',\phi'$ be the corresponding generators and Whitney disk.

\begin{enumerate}
\item If $\x,\y \in \mathfrak{G}(\SH_n)_- \cup \mathfrak{G}(\SH_n)_{twist}$ then
\[ n_z(\phi') = n_z(\phi) \quad n_w(\phi') = n_w(\phi) \quad \mu(\phi') = \mu(\phi)\]
\item if $\x,\y \in \mathfrak{G}(\SH_n)_+$ then
\[ n_z(\phi') = n_z(\phi) \quad n_w(\phi') = n_w(\phi) \quad \mu(\phi') = \mu(\phi)\]
\item If $\x \in \mathfrak{G}(\SH_n)_+$ and $\y \in \mathfrak{G}(\SH_n)_- \cup \mathfrak{G}(\SH_n)_{twist}$ then
\[ n_z(\phi') = n_z(\phi)-2 \quad n_w(\phi') = n_w(\phi) \quad \mu(\phi') = \mu(\phi) -2\]
\end{enumerate}
\end{lemma}

\begin{proof}
Let $D(\phi)$ be the domain in $\Sigma$ corresponding to $\phi$.  Orient $\gamma$ counterclockwise, as the boundary of the disk containing the basepoints $z_1,z_0$.  In the first 2 cases, the algebraic intersection of the $\beta$-components of $\del D(\phi)$ with $\gamma$ is 0.  Thus, the intersection numbers $n_z$ and $n_w$ and the Maslov index $\mu$ are unchanged by the Dehn twist.

To prove the third case, choose some $\x = (C,v_2,\dots,v_N) \in \mathfrak{G}(\SH_n)_-$ and let $\x_t = (c_k,v_2,\dots,v_N)$ where $k = \lfloor \frac{n}{2} \rfloor$.  There is a Whitney disk $\phi \in \pi_2(\x_t,\x)$ satisfying
\[n_z(\phi) = 1 \quad n_w(\phi) = 0 \quad \mu(\phi) = 1\]
The corresponding disk $\phi' \in \pi_2(\x'_t,\x')$ satisfies
\[n_z(\phi') = 3 \quad n_w(\phi) = 0 \quad \mu(\phi) = 3\]
The final case now follows from the first two and this observation since $n_z,n_w,$ and $\mu$ are additive under the composition
\[\pi_2(\x_1,\x_2) \times \pi_2(\x_2,\x_3) \rightarrow \pi_2(\x_1,\x_3)\]
\end{proof}

\begin{lemma}
\label{lemma:small-grading}
Let $\x \in \mathfrak{G}(\SH_n)$.  For $n$ sufficiently large, if $A(\x) \leq 1$ then $\x \in \mathfrak{G}(\SH_n)_- \cup \mathfrak{G}(\SH_n)_{twist}$.
\end{lemma}

\begin{proof}
Choose $\x,\y \in \mathfrak{G}(\SH_n)_+$.  By abuse of notation, let $\x,\y$ denote the corresponding generators in $\mathfrak{G}(\SH_{n'})_+$ for any $n' \geq 0$.  From Lemma \ref{lemma:homogeneous-grading}, we can conclude that $A(\x) - A(\y)$ is independent of $n$.  Moreover, since $\mathfrak{G}(\SH_n)_+$ is finite, there is some constant $K$, independent of $n$, such that $|A(\x) - A(\y)| < K$ for all $n$.  Similarly, there is some constant $L$ such that if $\x,\y \in \mathfrak{G}(\SH_n)_-$, then $|A(\x) - A(\y)| < L$ for any $n$.  However, if $\x \in \mathfrak{G}(\SH_n)_+$ and $\y \in \mathfrak{G}(\SH_n)_-$, then Lemma \ref{lemma:homogeneous-grading} also implies that $A(\x) - A(\y)$ grows without bound as $n$ limits to infinity.  In particular,  $A(\x) - A(\y) > 0$ when $n$ is sufficiently large.

Let $A_{max}$ denote the maximal Alexander grading of some generator $\x \in \mathfrak{G}(\SH_n)$.  We claim that if $n$ is sufficiently large then for all generators $\x$ satisfying $A(\x) = A_{max}$ then $\x \in \mathfrak{G}(\SH_n)_+$.  To prove this claim, suppose that $\x$ satisfies $A(\x) = A_{max}$.  If $n$ is sufficiently large, then $\x$ must be in $\mathfrak{G}(\SH_n)_{twist} \cup \mathfrak{G}(\SH_n)_+$ by the above argument.  Thus $\x$ has the form
\[\x = (V,v_2,\dots,v_n)\]
where $V \in \{C,D,a_1,b_1,c_1,d_1,\dots,a_k,b_k,c_k,d_k\}$ and $k = \lfloor \frac{n}{2} \rfloor$.  However, if $V \neq D$ define
\[\y = (D,v_2,\dots,v_n)\]
There is a domain $\phi \in \pi_2(\y,\x)$ with $n_z(\phi) > 0$ and $n_w(\phi) = 0$.  This contradicts the assumption that $A(\x) = A_{max}$.  Consequently, $V = D$ and $\x \in \mathfrak{G}(\SH_n)_+$.

Corollary \ref{cor:degree-Alex} implies that there exists some $k > 0$ such that $A_{max} \geq n- k$ for $n$ sufficiently large.  This further implies that if $\x \in \mathfrak{G}(\SH_n)_+$ then $A(\x) > n- k - K$.  Thus, for $n$ sufficiently large, no generator $\x$ with $A(\x) \leq 1$ can live in $\mathfrak{G}(\SH_n)_+$.
\end{proof}

Let $\CFKh(\SH_n)_{\leq j}$ be the subcomplex of $\CFKh(\SH_n)$ spanned by generators with Alexander grading at most $j$.  Let $\psi: \CFKh(\SH_n)_{\leq 1} \rightarrow \CFKh(\SH_{n+2})$ be the linear map which sends $\x$ to $\x'$ for any generator $\x \in \mathfrak{G}(\SH_n)$.

\begin{lemma}
Let $\x \in \mathfrak{G}(\SH_n)_- \cup \mathfrak{G}(\SH_n)_{twist}$.  Then $A(\psi \x) = A(\x) - 1$.
\end{lemma}

\begin{proof}
From Lemmas \ref{lemma:homogeneous-grading} and \ref{lemma:small-grading} we can conclude that $\psi$ preserves relative gradings.  Thus for all $\x$, the Alexander gradings satisfy $A(\psi \x) = A(\x) + a$ for some $a \in \ZZ$.

The Alexander grading shift follows from the computation of the Alexander polynomial in Corollary \ref{cor:twisting-Alex}.  Let $s$ be the minimal Alexander grading in which the Euler characteristic of $\CFKh(\SH_n)$ is nonzero.  The Euler characteristics of the $\CFKh(\SH_n)_{*,j}$ and $\CFKh(\SH_{n+2})_{*,j+a}$ agree if $j \leq 1$.  Thus $s + a$ is the minimal Alexander grading for which the Euler characteristic of $\CFKh(\SH_{n+2})$ is nonzero.  The Alexander polynomial computation implies that $s+a = s - 1$ and thus $a = -1$.
\end{proof}

\begin{proposition}
\label{prop:chain-bijection}
For $n$ sufficiently large, the map $\psi: \CFKh(\SH_n)_{\leq 1} \rightarrow \CFKh(\SH_{n+2})_{\leq 0}$ is a bijection of chain complexes.
\end{proposition}

\begin{proof}
From Lemmas \ref{lemma:homogeneous-grading} and \ref{lemma:small-grading} we can conclude that $\psi$ is an isomorphism of bigraded $\FF$-vector spaces and that it preserves relative gradings.  Thus, we just need to check that $\psi$ is a chain map.

Fix $\x,\y \in \mathfrak{G}(\SH_n)_- \cup \mathfrak{G}(\SH_n)_{twist}$ and some $\phi \in \pi_2(\x,\y)$ with $n_{\zs}(\phi) = 0$ and $\mu(\phi) = 1$.  We can choose an open neighborhood $W$ of $\overline{D}(\phi)$ to be disjoint from the curve $\gamma$ in Figure \ref{fig:crossing-heegaard}.  Thus, we can assume that the support of $T_{\gamma}$ is disjoint from $W$ and that the support of $T^*_{\gamma}$ is disjoint from $\Sym^N(W)$.  Genericity of paths of almost-complex structures is an open and dense condition.  Thus we can choose some path $J_s$ such that the moduli spaces $\cM_{J_s}(\phi_x)$ and $\cM_{J_s}(\phi_y)$ are transversely cut out for all choices of $\x_1,\x_2\in \TT_{\alpha} \cap \TT_{\beta}$; $\phi_x \in \pi_2(\x_1,\x_2)$; $\y_1,\y_2 \in \TT_{\alpha} \cap T^*_{\gamma} \TT_{\beta}$; $\phi_y \in \pi_2(\y_1,\y_2)$.  

Choose some $u \in \cM_{J_s}(\phi)$.  The Localization Principle \cite[Lemma 9.9]{Rasmussen} states that the image of $u$ is contained in $\Sym^N(W) \subset \Sym^N(S^2)$.  Since $T^*_{\gamma}$ is the identity on $\Sym^N(W)$, it is clear that $u \in \cM_{J_s}(\phi')$ as well.  Conversely, all maps $u' \in \cM_{J_s}(\phi')$ also lie in $\cM_{J_s}(\phi)$.  After quotienting by the $\RR$-action, this implies that $\# \widehat{\cM}_{J_s}(\phi)  = \# \widehat{\cM}_{J_s}(\phi')$.  It is now clear that $\psi(\widehat{\del} \x) = \widehat{\del} \psi \x$ for any $\x \in \CFKh(\SH_n)_{\leq 1}$.
\end{proof}

\subsection{$\HFKh$ computations}

\begin{proposition}
\label{prop:ranks-stabilization}
There exists some $k > 0$ such that for $n$ sufficiently large, the knot Floer homology of $\cL_n$ satisfies
\begin{align*}
\HFKh(\cL_n,j) &\cong \HFKh(\cL_{n+2},j+1) && \text{for $j \geq -k$}\\
\HFKh(\cL_{n},j) &\cong \HFKh(\cL_{n+2},j-1)[2] && \text{for $j \leq k$}
\end{align*}
where $[i]$ denotes shifting the homological grading by $i$.
\end{proposition}

\begin{proof}
Let $\psi$ be the map from Proposition \ref{prop:chain-bijection} and let $\psi^*$ be the induced map on homology.  Proposition \ref{prop:chain-bijection} implies that for some $m \in \ZZ$, the map
$$\psi^*: \HFKh_i(\cL_n,j) \rightarrow \HFKh_{i + m}(\cL_{n+2},j-1)$$ 
is an isomorphism for all $i \in \ZZ$ and $j \leq 1$.  To prove the second isomorphism of the proposition, we need to show that $m = -2$.

To compute the Maslov grading shift, we apply the skein exact triangle.  Fix $n$ sufficiently large.  Let $s$ be the minimal Alexander grading in which $\HFKh(\cL_n)$ is supported.  Thus $s-1$ is the minimal Alexander grading in which $\HFKh(\cL_{n+2})$ is supported.  Applying the skein exact sequence to the triple $(\cL_{n+2},\cL_n,\cL_{n+1})$, we can see that $$\widehat{f}: \HFKh(\cL_{n+1},s-1) \rightarrow \HFKh(\cL_{n+2},s-1)$$ is a graded isomorphism.  Moreover, since $\HFKh(\cL_{n+2},j) \cong \HFKh(\cL_{n},j) \cong 0$ for $j < s-1$, exactness implies that $s-1$ is also the minimal Alexander grading in which $\HFKh(\cL_{n+1})$ is supported.

Now apply the skein exact triangle to the triple $(\cL_{n+1},\cL_{n-1},\cL_n)$.  The modules $\HFKh(\cL_{n+1})$ and $\HFKh(\cL_n) \otimes V$ are supported in Alexander grading $s-1$ but $\HFKh(\cL_{n-1},s-1) \cong 0$.  Thus $$\widehat{f}: (\HFKh(\cL_{n}) \otimes V)_{*,s-1} \rightarrow \HFKh_*(\cL_{n+1},s-1)$$ is an isomorphism.  Combining the above two steps, this implies that there is a bigraded isomorphism $\HFKh(\cL_n,s)[2] \cong \HFKh(\cL_{n+2},s-1)$.  Thus, the grading shift must be $m = -2$.  This proves the second formula.

The first statement follows from the second using the symmetry 
\[\HFKh_i(\cL_n,j) \cong \HFKh_{i-2j}(\cL_n,-j)\]
\end{proof}

\begin{lemma}
\label{lemma:skein-maximal}
For $|n|$ sufficiently large, the skein exact triangle for the triple $(\cL_{n+1},\cL_{n-1},\cL_{n})$ is a split short exact sequence.  Consequently, either
\begin{align}
\label{eq:skein-formula-1}
\HFKh(\cL_n) & \simeq \HFKh(\cL_{n-1})[1] \oplus \HFKh(\cL_{n+1}), \text{ or} \\
\label{eq:skein-formula-2}
\HFKh(\cL_n) \otimes V & \simeq \HFKh(\cL_{n-1})[1] \oplus \HFKh(\cL_{n+1})
\end{align}
depending on whether the two strands through the twist region of $\cL_n$ lie in distinct components or the same component of the link.
\end{lemma}

\begin{proof}

Without loss of generality, we assume that $n$ is chosen so that in $\cL_{n}$, the two strands through the twist region lie in different components.  For any $j \in \ZZ$, the triangle inequality applied to the skein exact triangle proves that
\begin{align}
\label{eq:skein-bounded-1}
\rk \HFKh_i(\cL_n,j) &\leq \rk \HFKh_{i+1}(\cL_{n-1},j) + \rk \HFKh_{i}(\cL_{n+1},j) \\
\label{eq:skein-bounded-2}
 \rk (\HFKh(\cL_{n+1}) \otimes V)_{i-1,j}  & \leq \rk \HFKh_{i}(\cL_n,j) + \rk \HFKh_{i-1}(\cL_{n+2},j)
\end{align}

Suppose that $j \leq 0$.  Then applying Proposition \ref{prop:ranks-stabilization} to $|n|$ sufficiently large, we can conclude that
\begin{align}
\label{eq:rank-times-v}
\rk (\HFKh(\cL_{n+1}) \otimes V)_{i-1,j} &= \rk \HFKh_{i+1}(\cL_{n+1},j+1) + 2 \cdot \rk \HFKh_{i}(\cL_{n+1},j) \\ 
& \, \,+ \rk \HFKh_{i-1}(\cL_{n+1},j-1) \nonumber \\
\label{eq:stabilize-replace}
 &= \rk \HFKh_{i-1}(\cL_{n+3},j) +  \rk \HFKh_{i}(\cL_{n+1},j) \\ 
& \, \,+ \rk \HFKh_{i}(\cL_{n+1},j) + \rk \HFKh_{i+1}(\cL_{n-1},j) \nonumber \\
\label{eq:reverse-inequality}
& \geq  \rk \HFKh_{i-1}(\cL_{n+2},j) + \rk \HFKh_{i}(\cL_n,j)
\end{align}
Equation \ref{eq:rank-times-v} follows from the definition of $V$ and Equation \ref{eq:stabilize-replace} can be obtained from Equation \ref{eq:rank-times-v} using Proposition \ref{prop:ranks-stabilization}.  Finally, applying Inequality \ref{eq:skein-bounded-1} twice yields Inequality \ref{eq:reverse-inequality}.  Combining Inequalities \ref{eq:skein-bounded-2} and \ref{eq:reverse-inequality} proves that
\begin{align}
\label{eq:case2-equal}
\rk \HFKh_{i}(\cL_n,j) + \rk \HFKh_{i-1}(\cL_{n+2},j) = \rk (\HFKh(\cL_{n+1}) \otimes V)_{i-1,j}
\end{align}
Furthermore, the symmetry $\rk \HFKh_{i-2j}(\cL_n,-j) = \rk \HFKh_i(\cL_n,j)$ implies that Equation \ref{eq:case2-equal} holds for all $j \in \ZZ$.  This proves that the rank of $\widehat{g}$ is 0 in the skein exact triangle for the triple $(\cL_{n+2},\cL_{n},\cL_{n+1})$.

Now we consider the triple $(\cL_{n+1},\cL_{n-1},\cL_n)$.  For $j \leq 0$, we can conclude that

\begin{align}
\label{eq:v-tensor-expand}
\left( V \otimes \left( \HFKh(\cL_{n-1})[1] \oplus \HFKh(\cL_{n+1}) \right) \right)_{i,j} &\simeq \HFKh_{i-2}(\cL_{n-2},j) \oplus \HFK_{i-1}(\cL_{n},j)\\
& \, \, \oplus \HFKh_{i-1}(\cL_{n},j) \oplus \HFKh_i(\cL_{n+2},j) \nonumber \\
\label{eq:v-tensor-replace}
&\simeq \HFKh_{i}(\cL_{n},j+1) \oplus \HFK_{i-1}(\cL_{n},j)\\
& \, \, \oplus \HFKh_{i-1}(\cL_{n},j) \oplus \HFKh_{i-2}(\cL_{n},j-1) \nonumber \\
\label{eq:v-tensor-def}
& \simeq (V \otimes \HFKh(\cL_n))_{i,j}
\end{align}
The isomorphism in Line \ref{eq:v-tensor-expand} follows from the isomorphism in Line \ref{eq:skein-formula-2}, which has already been proven.  Line \ref{eq:v-tensor-replace} is obtained by applying Proposition \ref{prop:ranks-stabilization} and finally Line \ref{eq:v-tensor-def} follows from the definition of $V$.   Thus
\[V \otimes \left( \HFKh(\cL_{n-1})[1] \oplus \HFKh(\cL_{n+1}) \right) \cong V \otimes \HFKh(\cL_n)\]
Removing the $V$ factors on both sides proves the statement for $j \leq 0$ and the statement for $j \geq 0$ follows from symmetry.
\end{proof}

\begin{proof}[Proof of Theorem \ref{thrm:stabilization}]
The first statement is Proposition \ref{prop:ranks-stabilization}.

To prove the second, fix some $n_0$ sufficiently large so that Part (1) applies.  Set $k = n_0 + 1$ and
\begin{align*}
\widehat{F}_{\circ} &:= \bigoplus_{j = -\infty}^{-2} \HFKh(\cL_{n_0},j+2)[-2] & \widehat{A} &:= \HFKh(\cL_{n_0},-1)[-1] \\
\widehat{F}_{\bullet} &:= \bigoplus_{j=2}^{\infty} \HFKh(\cL_{n_0},j-2)[2] & \widehat{B} &:= \HFKh(\cL_{n_0},0)
\end{align*}

Then the statement holds for $n = n_0$.  Applying Part (1) inductively for $n > n_0$ proves Part (2).

Finally, Part (3) follows from Part (2) and Lemma \ref{lemma:skein-maximal}.
\end{proof}

%%%%%%%%%%%%%%%%%%%%%%%%%%%%%%%%%%%%%%%%%%%%%%%
\section{Positive mutants}
%%%%%%%%%%%%%%%%%%%%%%%%%%%%%%%%%%%%%%%%%%%%%%%

Throughout this section, let $\cL$ be an oriented link with an essential Conway sphere as in Figure \ref{fig:Lkl}.  Let $\cL_{k,l}$ denote the link obtained by adding $k$ half-twists above and $l$ half-twists below.  Note that applying two flypes along the horizontal axis to the tangle $\mathcal{T}$ is an isotopy between $\cL_{k,l}$ and $\cL_{k+2,l-2}$.  Thus, for any $k,l,i \in \ZZ$, the link $\cL_{k+2i,l-2i}$ is isotopic to $\cL_{k,l}$.  Furthermore, $\cL_{k+1,l-1}$ can be obtained from $\cL_{k,l}$ by the positive mutation and then a flype.

%%%%%%%%%%%%%%%%%%%%%%%%%%%%%%%%%%%%%%%%%%%%%%%
\begin{figure}
\centering
\labellist
	\huge\hair 2pt
	\pinlabel $\mathcal{T}$ at 565 85
	\pinlabel $l$ at 395 85
	\pinlabel $k$ at 760 85
\endlabellist
\includegraphics[width=.6\textwidth]{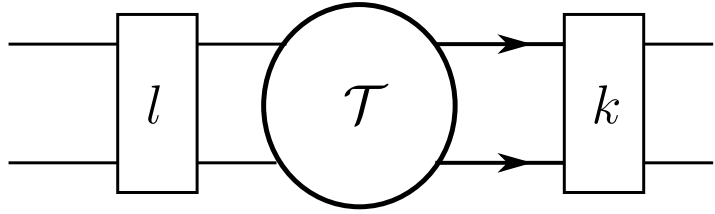}
\caption{The link $\cL_{k,l}$ near the Conway sphere containing the tangle $\mathcal{T}$.}
\label{fig:Lkl}
\end{figure}
%%%%%%%%%%%%%%%%%%%%%%%%%%%%%%%%%%%%%%%%%%%%%%%%%

\subsection{Bigraded invariance}

We prove the first part of Theorem \ref{thrm:mutants-isom-hfk} in this subsection and leave the second piece to the following subsection.  In addition, we then use Theorem \ref{thrm:mutants-isom-hfk} to prove Theorem \ref{thrm:pos-mutants-hfk}.

\begin{theorem}
\label{thrm:mutants-isom}
For $|n|$ sufficiently large, the mutants $\cL_{n,0}$ and $\cL_{n+1,-1}$ have isomorphic knot Floer homology
\[\HFKh(\cL_{n,0}) \simeq \HFKh(\cL_{n+1,-1})\]
\end{theorem}

\begin{proof}
Let $|n|$ be sufficiently large.  Suppose that at the $n^{\text{th}}$-crossing above the mutation sphere, the two strands of $\cL_{n,0}$ lie in different components. By Lemma \ref{lemma:skein-maximal}, the knot Floer homology for the two mutant links are given by
\begin{align*}
\HFKh(\cL_{n,0}) &\simeq \HFKh(\cL_{n-1,0}) [1] \oplus \HFKh(\cL_{n+1,0}) \\
 \HFKh(\cL_{n+1,-1}) &\simeq \HFKh(\cL_{n+1,-2})[1] \oplus \HFKh(\cL_{n+1,0})
\end{align*}
The statements now follows from the fact that $\cL_{n-1,0}$ and $\cL_{n+1,-2}$ are isotopic.

Similarly, if at the $n^{\text{th}}$-crossing the two strands lie in the same component, then
\begin{align*}
\HFKh(\cL_{n,0}) \otimes V &\simeq \HFKh(\cL_{n-1,0}) [1] \oplus \HFKh(\cL_{n+1,0}) \\
 \HFKh(\cL_{n+1,-1}) \otimes V &\simeq \HFKh(\cL_{n+1,-2})[1] \oplus \HFKh(\cL_{n+1,0})
\end{align*}
and it is clear that $\HFKh(\cL_{n,0})$ and $\HFKh(\cL_{n+1,-1})$ are isomorphic.
\end{proof}

We can now use the Kinoshita-Terasaka and Conway knots to prove Theorem \ref{thrm:pos-mutants-hfk}.  Let $KT$ denote the Kinoshita-Terasaka knot ($11n42$) and let $C$ denote the Conway knot ($11n34$).  Figure \ref{fig:KTgamma} contains a diagram of $KT$.  The knots $KT$ and $C$ are positive mutants and thus Theorem \ref{thrm:mutants-isom} applies.  Let $\gamma$ be the curve on the Conway sphere for $KT$ that is fixed by the involution of $S^2$ corresponding to the positive mutation.  Let $\widetilde{KT}$ and $\widetilde{C}$ denote links given by the unions of $KT$ and $C$, respectively, with $\gamma$.

%%%%%%%%%%%%%%%%%%%%%%%%%%%%%%%%%%%%%%%%%%%%%%%
\begin{figure}
\centering
\labellist
	\small\hair 2pt
	\pinlabel $S_1$ at 50 150
	\pinlabel $a_1$ at 80 122
	\pinlabel $a_2$ at 15 30
	\pinlabel $b_1$ at 120 0
	\pinlabel $b_2$ at 110 122
	\pinlabel $S_2$ at 215 150
\endlabellist
\includegraphics[width=.45\textwidth]{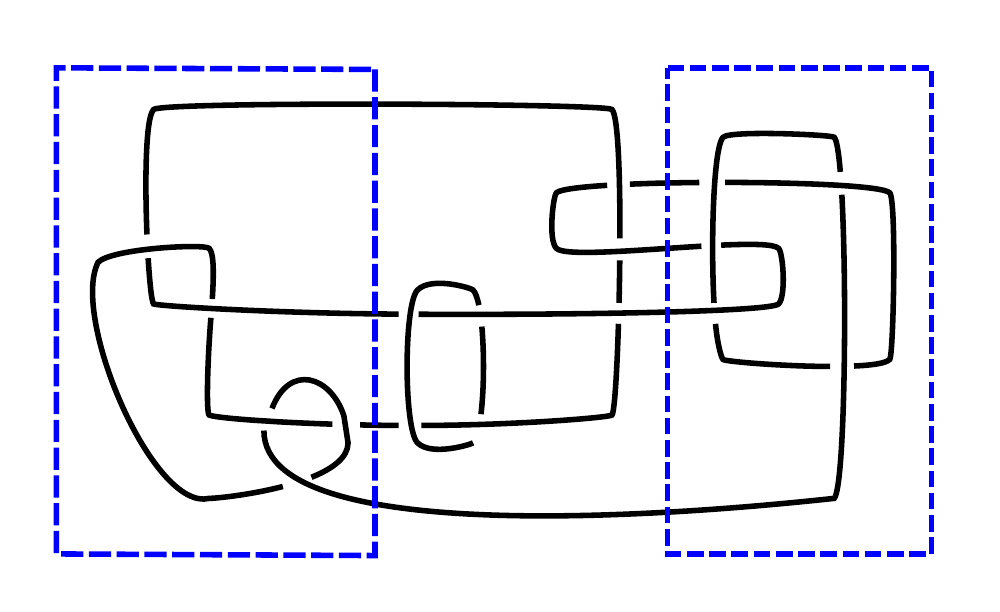}
\caption{The Kinoshita-Terasaka knot 11n42 linked with an unknotted curve $\gamma$.  The two Conway spheres are marked in dotted outline and the arcs of $KT$ cut out by $S_1$ are labeled $a_1,a_2,b_1,b_2$.}
\label{fig:KTgamma}
\end{figure}
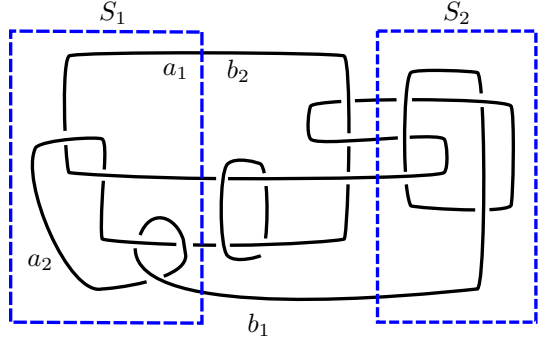
%%%%%%%%%%%%%%%%%%%%%%%%%%%%%%%%%%%%%%%%%%%%%%%%%

Let $S_1,S_2$ denote the Conway spheres in Figure \ref{fig:KTgamma}.  Then $S_1$ decomposes $KT$ into the tangles $T_1,T_2$ while $S_2$ decomposes $KT$ into $R_1,R_2$.  The tangle $T_1$ consists of two arcs $a_1,a_2$ with $a_1$ unknotted and $a_2$ an arc with a right-handed trefoil tied in.  The tangle $T_2$ consists of two arcs $b_1,b_2$, each unknotted.  Similarly, the tangle $R_1$ consists of two arcs $c_1,c_2$, with $c_1$ unknotted and $c_2$ an arc with a left-handed trefoil tied in, while $R_2$ consists of two unknotted arcs $d_1,d_2$.  We can obtain $C$ from $KT$ by mutating the tangle $T_2$.  Note that, in order, the link $KT$ is $a_1,b_1,a_2,b_2$, while the link $C$ is $a_1,b_2,a_2,b_1$.  

\begin{lemma}
There exist exactly 2 essential Conway spheres for $KT$ and $C$.
\end{lemma}

\begin{proof}
Suppose there is a third essential Conway sphere $S$.  The tangles $T_1,T_2$ are prime, so we can assume that either (a) $S$ lies completely within $T_1$ or $T_2$, or (b) it intersects both $T_1$ and $T_2$ along essential 2-punctured disks.  Let $\left[ \frac{p}{q} \right]$ denote the rational tangle determined by the continued fraction expansion of $\frac{p}{q}$.  In the algebra of tangles, we can express the tangles of $KT$ as
\begin{align*}
T_1 &= \left[ -\frac{1}{3} \right] + \left[\frac{1}{2}\right] & T_2 &= \left[ 2 \right] \ast \left( \left[ \frac{1}{3} \right] + \left[ -\frac{1}{2} \right] \right)
\end{align*} 

The tangle $T_1$ contains a unique essential disk, where the rational tangles $[\frac{1}{3}]$ and $[\frac{1}{2}]$ were joined, and no essential Conway spheres.  Similarly, the tangle $T_2$ contains a unique essential disk where the tangles $[2]$ and $R_1$ were joined and on essential spheres.  However, the boundaries of these two disks are not isotopic.  Specifically, the boundary of the disk in $T_1$ separates the endpoints of $a_1$ from the endpoints of $b_1$.  The boundary of the disk in $T_2$ separates the endpoints of $b_1$ from the endpoints of $b_2$.  

Finally, there are the same two corresponding disks in $C$ with the same boundaries up to isotopy.
\end{proof}

Let $KT_{n}$ be the knot obtained by performing $-\frac{1}{n}$-surgery to $\gamma$ and let $C_{n}$ denote the corresponding positive mutant.  Diagrammatically, this surgery corresponds to applying $2n$ half-twists just outside the mutation sphere.  The mutants $KT_{n}$ and $C_n$ have the same Conway spheres.  In particular, each knot $KT_n$ is obtained from the union of $T_1$ and $T_2$, however the gluing map that identifies $S^2$ with $S^2$ is modified by $n$ Dehn twists along $\gamma$.

\begin{align}
KT_n &= T_1 \cup_{\phi_n} T_2 & C_n &= T_1 \cup_{\phi_n \circ \tau} T_2
\end{align}

\begin{lemma}
There exist exactly 2 essential Conway spheres for $KT_n$ and $C_n$ for all $n \in \ZZ$.
\end{lemma}

\begin{proof}
$KT_n$ and $C_n$ are comprised of the same pair of tangles $T_1,T_2$ as $KT$ and $C$.  Thus, there are unique essential disks with boundaries in the mutation sphere.  However, the boundaries of the disks determine different partitions of the 4 points where the knot intersects $S_1$.  Thus, the Dehn twists along $\gamma$ cannot match up the boundaries to give a closed sphere.
\end{proof}

To distinguish $KT_n$ from $C_n$, we will adopt a strategy similar to the one in \cite{Cochran-Ruberman} \footnote{I would also like to thank Chuck Livingston for providing an unpublished draft of \cite{Cochran-Ruberman}.}.  Let $L$ be an ordered, oriented 2-component link and label the components $x$ and $y$.  When $L$ has linking number 0, Cochran \cite{Cochran} defined a sequence of higher-order linking invariants $\beta^i_{x},\beta^i_y$ for $i \geq 0$.  The invariants $\beta^1_x(L) = \beta^1_y(L)$ are the Sato-Levine invariant.  The higher invariants are defined inductively by taking `derivatives' of the original link L as follows.  Since $lk(x,y) = 0$, we can choose a Seifert surface $F$ for $x$ disjoint from $y$ and similarly a Seifert surface $G$ for $y$ disjoint from $x$.  We can further assume that $F$ and $G$ intersect transversely along a knot $K$.  The `partial derivatives' of $L$ are the links
\begin{align}
D_x(L) &:= x \cup K & D_y(L) &:= y \cup K
\end{align}
The derived links also have linking number 0 and therefore the process can be iterated.  Cochran's higher order linking invariants are then  defined inductively by
\begin{align}
\beta^{i+1}_x(L) &:= \beta^i_x(D_x(L)) & \beta^{i+1}_y(L) &:= \beta^i_y(D_y(L))
\end{align}
In general, these higher order invariants are manifestly non-symmetric in $x$ and $y$.

In \cite{Cochran-Ruberman}, Cochran and Ruberman apply the higher-order linking invariants to define invariants of tangles.  Let $T$ be a tangle consisting of an ordered pair of two disjoint arcs in $B^3$.  Let $C(T)$ be any 2-component rational closure of $T$.  That is, $C(T)$ is the union of $T$ and $(B^3,T_0)$, the tangle consisting of two boundary-parallel arcs.  The difference
\begin{align}
I^i(T) &:= \beta^i_x(C(T)) - \beta^i_y(C(T))
\end{align}
is well-defined and independent of the choice of rational closure $C(T)$ \cite[Theorem 4.1]{Cochran-Ruberman}.  Therefore it defines an invariant of the tangle T for each $i \geq 0$.   Reversing the ordering of the arcs of $T$ changes the sign of $I^i(T)$ for all $i$.  Thus, it follows that if there is a diffeomorphism of $T$ that exchanges the two arcs, then $I^i(T) = 0$ for all $i \geq 0$ \cite[Lemma 4.3]{Cochran-Ruberman}.

%%%%%%%%%%%%%%%%%%%%%%%%%%%%%%%%%%%%%%%%%%%%%%%
\begin{figure}
\centering
\labellist
	\small\hair 2pt
	\pinlabel $K$ at 190 400
	\pinlabel $X$ at 120 540
	\pinlabel $Y$ at 430 560
\endlabellist
\includegraphics[width=.45\textwidth]{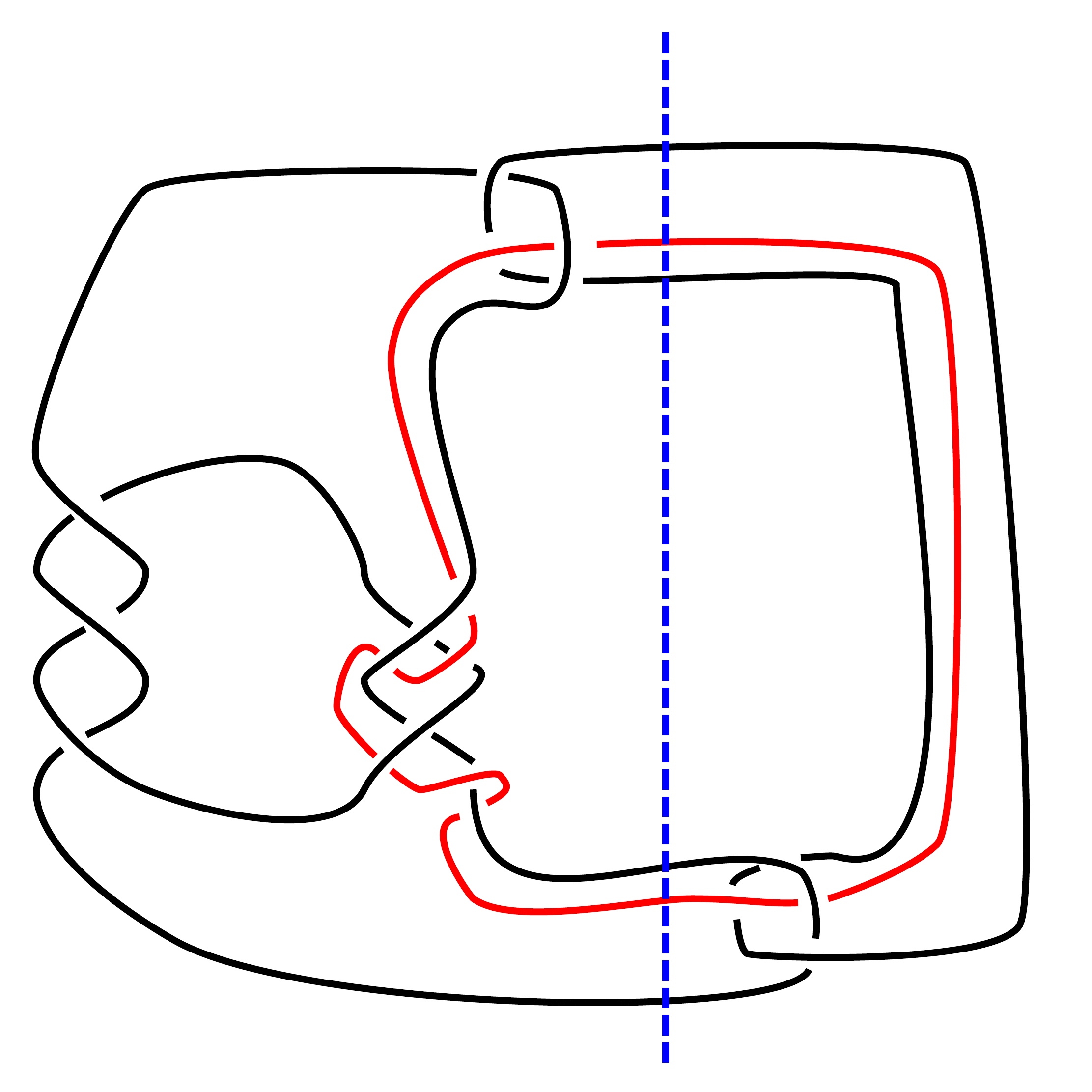}
\caption{A rational closure $L$ of the tangle $T_2$.}
\label{fig:Qclosure}
\end{figure}
%%%%%%%%%%%%%%%%%%%%%%%%%%%%%%%%%%%%%%%%%%%%%%%%%

\begin{lemma}
Let $T_2$ and $R_2$ be the tangles for the KT knots.  Then $I^2(T_2)$ and $I^2(R_2)$ are nonzero.
\end{lemma}

\begin{proof}
$R_2$ is the mirror of $T_2$, so it suffices to prove the statement for $T_2$.  Take the rational closure $L$ of $T_2$ in Figure \ref{fig:Qclosure} and let $x$ and $y$ denote the components labeled in the figure.  Both components are unknots and the linking number is 0.   Let $F$ and $G$ be the Seifert surfaces of $x$ and $y$, respectively, obtained by Seifert's algorithm.  Then $F$ intersects $y$ in two points.  We can remove this intersection by tubing between the intersection points along an arc of $y$ connecting them.  Let $F'$ denote this new surface.  Similarly, $G$ intersects $x$ in two points and by adding a tube we can choose $G'$ disjoint from $x$.  Then $K = F' \pitchfork G'$ is the union of 0-framed pushoffs of the arcs in $x$ and $y$ corresponding to the tubes.  See Figure \ref{fig:Qclosure}.

The link $D_x(L) = x \cup K$ is isotopic to the Whitehead link, while the link $D_y(L)$ is the unlink.  Let $W$ denote $D_x(L)$ and fix an orientation on $W$.  Let $V$ be obtained by changing one positive crossing of $W$ to a negative crossing and let $Z$ be the oriented 0-resolution of $W$ at this crossing.  The oriented resolution splits $x$ into two components $x_1,x_2$.  Let $Z_1,Z_2$ be the 2-component sublinks of $Z$ consisting of $x_1$ and $K$ or $x_2$ and $K$, respectively.  Then the crossing change formula for the Sato-Levine invariant implies that
\[\beta^1(V) - \beta^1(W) = \text{lk}(Z_1)\cdot \text{lk}(Z_2)\]
The linking numbers of $Z_1$ and $Z_2$ are nonzero while $V$ is the unlink, so $\beta^1(W) = \beta^2_x(L) \neq 0$.  However, $D_y(L)$ is the unlink and so $\beta^2_Y(L) = 0$.  Thus $I^2(T_2) = \beta^2_x(L) - \beta^2_y(L) \neq 0$.  
\end{proof}

\begin{proposition}
\label{prop:KTC-nonisotopic}
For all $n$, the knots $KT_n$ and $C_n$ are not isotopic.
\end{proposition}

\begin{proof}
Suppose there is an isotopy of $KT_n$ to $C_n$.  We can assume that this isotopy acts transitively on the set of Conway spheres.  Since there are two such spheres, the isotopy either fixes them or swaps them.

If the isotopy swaps the Conway spheres, then the arc $a_2$ must be sent to one of the arcs $c_1,c_2,d_1,d_2$.  However, the arc $a_2$ contains a right-handed trefoil and none of the latter four arcs do.  Thus, it is impossible for an orientation-preserving diffeomorphism to send $S_1$ to $S_2$ and the isotopy must fix the Conway spheres.

Of the four arcs $a_1,a_2,b_1,b_2$, only $a_2$ is knotted so the isotopy must send $a_2$ to itself and therefore $a_1$ to itself as well.  Consequently the isotopy must send the tangles $T_1,T_2$ to themselves.  However, the isotopy now must exchange $b_1$ and $b_2$ since the mutation exchanged these arcs.  This implies there is a diffeomorphism from $T_2$ to itself that exchanges the arcs.  But this is a contradiction since $I^2(T_2) \neq 0$.
\end{proof}

\begin{remark}
Proposition \ref{prop:KTC-nonisotopic} also shows that the standard Kinoshita-Terasaka and Conway knots are not isotopic.
\end{remark}

\begin{proof}[Proof of Theorem \ref{thrm:pos-mutants-hfk}]
Let $\widehat{KT}$ be the union of $KT$ with $\gamma$ and let $\widehat{C}$ be the union of $C$ with $\gamma$.  A computation in SnapPy \cite{SnapPy} shows that the links $\widehat{KT}$ and $\widehat{C}$ are hyperbolic with volume $\sim 23.975$.  Thus, by the hyperbolic Dehn surgery theorem, surgery on $\gamma$ with slope $-\frac{1}{n}$ is hyperbolic for all but finitely many values of $n$.  Consequently, for $|n|$ sufficiently large the knot $KT_n$ is hyperbolic and thus prime.  Again for $|n|$ sufficiently large, Theorem \ref{thrm:mutants-isom} states that there is a bigraded isomorphism
\[\HFKh(KT_{n}) \cong \HFKh(C_{n})\]
Thus $\{(KT_n,C_n)\}$ for $|n| \gg 0$ is the required family of prime mutants. 
\end{proof}

\begin{remark}
It is possible that $|n| = 1$ is sufficiently large for the family in Theorem \ref{thrm:pos-mutants-hfk}.  A computation with the {\it py{\_}hfk} Python module \cite{py-hfk} shows that all four of the knots $KT_{1},C_1,KT_{-1},C_{-1}$ are $\HFKh$ homologically thin.  Thus, the Alexander polynomial determines the knot Floer homology and
\[\HFKh(KT_1) \cong \HFKh(C_1) \qquad \HFKh(KT_{-1}) \cong \HFKh(C_{-1})\]
\end{remark}

\subsection{Concordance invariants}

The knot Floer group $\HFKm(K)$ contains a well-known concordance invariant $\tau(K)$ \cite{OS-4ball}.  The free part $\HFKm(K)/\text{Tors}$ is isomorphic to the polynomial ring $\FF[U]$ and $-\tau(K)$ is the maximal grading of a nontorsion element.  More specifically,
\[\tau(K) :=  - \text{max}\left \{A(\x)) : \x \text{ is not $U^k$-torsion for any $k > 0$}\right\}  \]
More generally, if $L$ is a 2-component link, we can choose a pair of elements $x_1,x_2$ with homogeneous bigradings that generate $\HFKm(L)/\text{Tors} \cong \FF[U]^2$.  The {\it $\tau$-set} of $L$ is the set $\tau(L) = \{-A(\x_1),-A(\x_2)\}$.

\begin{lemma}
\label{lemma:tau-facts}
 Let $K$ be a knot, let $L$ a 2-component link, and suppose $K$ can be obtained from $L$ by an elementary merge cobordism.  Then
\[\tau(K) \in \tau(L)\]
\end{lemma}

\begin{proof}
The statement follows easily from the following three inequalities, proven in \cite[Chapter 8]{OSS-book}.  Since $L$ and $K$ are related by an elementary saddle move, their $\tau$ values satisfy the inequalities
\[\tau(K) - 1 \leq \tau_{min}(L) \leq \tau(K) \qquad \tau(K) \leq \tau_{max}(L) \leq \tau(K) + 1\]
In addition, since $L$ has 2 components, the maximum and minimum values satisfy
\[\tau_{max}(L) - \tau_{min}(L) \leq 2-1 = 1\]
\end{proof}

\begin{theorem}
\label{thrm:tau-invariance}
Let $K$ be a knot with an essential Conway sphere as in Figure \ref{fig:Lkl} and let $K' = K_{1,-1}$ be its positive mutant.  For all $n \in \ZZ$, the knots $K_{2n,0}$ and $K'_{2n,0} = K_{2n+1,-1}$ are positive mutants and for $|n| \gg 0$, we have
\[\tau(K_{2n,0}) = \tau(K'_{2n,0})\]
\end{theorem}

\begin{proof}
From Theorem \ref{thrm:stabilization}, we can conclude that there exist integers $a,b$ such that for $n \gg 0$, if $\HFKh_0(K_{2n,0},s)$ is nontrivial, then $s$ lies in the interval $[a + n,b + n]$.  Moreover, Theorem \ref{thrm:mutants-isom} implies that $\HFKh_0(K'_{2n,0},s)$ is supported in the same interval of Alexander gradings.  

Set $c_n := \tau(K_{2n,0}) - n$ and consider the sequence $\{c_n\}$ for $n \gg 0$.  Since $\tau(K_{2n+2,0}) - \tau(K_{2n,0}) \leq 1$, the sequence is monotonically decreasing.  It is also bounded from below by $a$ and thus $\lim_{n \rightarrow \infty} c_n = C$ exists.  Define the sequence $\{c'_n\}$ and limit $C'$ similarly for the family $\{K'_{2n,0}\}$.  

Choose $n_0$ so that $\tau(K_{2n,0}) = C + n$ and $\tau(K'_{2n,0}) = C' + n$ for all $n \geq n_0$.  Set $L_n = K_{2n+1,0}$.  There are elementary merge cobordisms from $L_n$ to $K_{2n,0}$ and $K_{2n+2,0}$ given by resolving and introducing a positive crossing, respectively.  Thus $\tau(L_n) = \{C+n,C+n+1\}$.  There are also elementary merge cobordisms from $K_n$ to $K_{2n+1,1} = K'_{2n+2,0}$ and $K_{2n+1,-1} = K'_{2n}$ given by introducing a positive and negative crossing, respectively.   Thus, $\tau(L_n) = \{C' + n, C' + n + 1\}$.  Clearly $C = C'$ and the statement for $n > 0$ follows immediately.  The statement for $n < 0$ follows by an identical argument.
\end{proof}

\begin{proof}[Proof of Theorem \ref{thrm:mutants-isom-hfk}]
The theorem is a combination of Theorems \ref{thrm:mutants-isom} and \ref{thrm:tau-invariance}, along with the easy corollary that $\cL_n$ and $\cL'_n$ have the same genera because the Seifert genus is exactly the highest Alexander grading supporting $\HFKh$.
\end{proof}

The following proposition is also an easy consequence of Lemma \ref{lemma:tau-facts}.  We do not need it to prove Theorem \ref{thrm:tau-invariance} but it may be interesting in its own right.

\begin{proposition}
Let $K$ be a knot with an essential Conway sphere as in Figure \ref{fig:Lkl} and let $K' = K_{1,-1}$ be its positive mutant.  Furthermore, set $L = K_{1,0} = K'_{0,1}$ and $L' = K_{0,1} = K'_{1,0}$.  Then either
\[\tau(K) = \tau(K')\]
or 
\[\tau(L) = \tau(L')\]
\end{proposition}

\begin{proof}
Resolving a positive crossing gives an elementary merge cobordism from $L = K_{1,0}$ to $K_{2,0}$.  Introducing a negative crossing gives an elementary merge cobordism from $L_n = K_{1,0}$ to $K_{1,-1} = K'$.  There are similar elementary merge cobordisms from $L'$ to $K$ and $K'$.  Thus, by Lemma \ref{lemma:tau-facts},
\[\tau(K),\tau(K') \in \tau(L) \qquad \tau(K),\tau(K') \in \tau(L')\]
Moreover, the sets $\tau(L)$ and $\tau(L')$ have at most 2 elements.  Thus, if $\tau(K) \neq \tau(K')$ then $\tau(L) = \tau(L') = \{\tau(K),\tau(K')\}$.  
\end{proof}

%%%%%%%%%%%%%%%%%%%%%%%%%%%%%%%%%%%%%%%%%%%%%%%%%%%%%%%
\bibliographystyle{alpha}
\nocite{*}
\bibliography{References}

%%%%%%%%%%%%%%%%%%%%%%%%%%%%%%%%%%%%%%%%%%%%%%%%%%%%%%%

\end{document}